\documentclass[11pt,letterpaper]{article}
\usepackage[utf8]{inputenc}
\usepackage[english]{babel}
\usepackage{amsmath}
\usepackage{amsfonts}
\usepackage{amssymb}
\usepackage{makeidx}
\usepackage{graphicx}
\usepackage[left=2cm,right=2cm,top=2cm,bottom=2cm]{geometry}
\usepackage{lipsum}
\usepackage{amsthm}
\usepackage[colorlinks=true, linkcolor=blue]{hyperref}
\usepackage{enumerate}
\usepackage{mathrsfs}

\newtheorem{theorem}{Theorem}[section]
\newtheorem{lemma}{Lemma}[section]
\newtheorem{corollary}{Corollary}[section]
\newtheorem{proposition}{Proposition}[section]

\newtheorem{rmk}{Remark}[section]

\begin{document}
\begin{center}
\Large{\textbf{Strong in-domatic number in digraphs.}}
\end{center}
\[
\]

\begin{center}
Laura Pastrana-Ram\'irez \footnote{This research did not receive any specific grant from funding agencies in the public, commercial, or not-for-profit sectors.} (lauparra27@gmail.com),
\\ Roc\'io S\'anchez-L\'opez \footnote{This research did not receive any specific grant from funding agencies in the public, commercial, or
not-for-profit sectors.} (usagitsukinomx@yahoo.com.mx),
\\  Miguel Tecpa-Galv\'an \footnote{ This research did not receive any specific grant from funding agencies in the public, commercial, or
not-for-profit sectors.} (Corresponding author) (miguel.tecpa05@gmail.com).
\end{center}

\begin{center}
\textit{Facultad de Ciencias, Universidad Nacional Autónoma de México
Circuito Exterior s/n, Coyoacán, Ciudad Universitaria, 04510, Ciudad de México, CDMX.}
\end{center}

\begin{abstract}
Let $D=(V,A)$ be a digraph and $\mathfrak{S}$ a partition of $V(D)$. We say that $\mathfrak{S}$ is a strong in-domatic partition if every $S$ in $\mathfrak{S}$ holds that every vertex not in $S$ has at least one out-neighbor in $S$, that is $S$ is an in-dominating set, and $D\langle S \rangle$ is strongly connected. The maximum number of elements in a strong in-domatic partition is called the strong in-domatic number of $D$ and it is denoted by $\mathsf{d}_{s}^{-}(D)$. In this paper we introduce those concepts and determine the value of $\mathsf{d}_{s}^{-}$ for semicomplete digraphs and planar digraphs.  We  show some structural properties of digraphs which have a strong in-domatic partition and we  see some bounds for  $\mathsf{d}_{s}^{-}(D)$. Then we  study this concept in the Cartesian product, composition, line digraph and other associated digraphs.

In addition, we  characterize strong in-domatic critical digraphs and we  give two families strong in-domatic critical digraphs which hold some properties, where a strong in-domatic critical digraph $D$ holds that  $\mathsf{d}_{s}^{-}(D-e) = \mathsf{d}_{s}^{-}(D) -1 $ for every $e$ in $A(D)$. 
\end{abstract}


\textbf{MSC clasification.} 05C20, 05C40, 05C69, 05C76.

\section{Introduction}

Let $G$ be a graph, a \textit{domatic partition of $V(G)$}, say $\mathfrak{S}$, is a partition of $V(G)$ such that for every $S$ in $\mathfrak{S}$, $S$ is a dominating set. The maximum number of elements in such partition is called the  \textit{domatic number of $G$}, denoted by $\mathsf{d}(G)$. This concept was introduced by Cockayne and Hedetniemi in \cite{2}. In \cite{complex2}, Garey and Johnson showed that, for every $k \geq 3$, determine whether or not the domatic number of a given graph is $k$ is NP-complete.
In \cite{complex}, Poon, Yen and Ung proved that finding a domatic partition into 3 dominating sets is NP-complete on planar bipartite graphs, and finding a domatic partition with $\mathsf{d}(G)$ elements in co-bipartite graphs is NP-complete. In \cite{complex1}, Riege, Rothe, Spakowski and Yamamoto showed that, given an arbitrary graph $G$, it is possible to determine if $V(G)$ can be partitioned into 3 disjoint dominating sets with a deterministic algorithm in time $2.695^{n}$ (up to polynomial factors) and in polynomial space.

Domatic partitions in graphs have been studied for some researches due its applications and theoretical results (see \cite{ejem3}, \cite{complex2}, \cite{ejem2} \cite{ejem1} \cite{complex},  \cite{complex1}). Due a large amount of kinds of dominating sets  (see \cite{16} and \cite{4}), several authors defined variants on the domatic number in graphs, for instance, total domatic number (Cockayne, Hedetniemi and Dawes \cite{total}), idomatic number (Cockayne and Hedetniemi \cite{2}), $k$-domatic number (Zelinka \cite{kdomatic}) and tree domatic number (Chen \cite{tree}). 
In the same spirit, Laskar and Hedetniemi \cite{5} introduced the \textit{connected domatic number} as follows: for a digraph $G$ a \textit{connected domatic partition of $V(G)$} is a domatic partition such that every element in such partition induces a connected graph in $G$. The maximum number of elements in a connected domatic partition of $V(G)$ is the \textit{domatic connected number of $G$},  denoted by $\mathsf{d}_{c}(G)$. Whenever a graph $G$ holds that $\mathsf{d}_{c} (G-a) < \mathsf{d}_{c} (G)$ for every edge $a$ of $G$, it is said that $G$ is a  \textit{connectively domatically critical graph}. Such graphs were introduced and characterized by Zelinka in  \cite{6}.
Hartnell and Rall studied the connected domatic number in planar graphs \cite{3}, in particular, they showed the following results.

\begin{theorem}\cite{3}
\label{PlanarT1}
Let $G$ be a planar graph. The connected domatic number of $G$ is at most $4$, and $K_{4}$ is the only planar graph achieving this bound.
\end{theorem}

\begin{theorem}\cite{3}
\label{PlanarT2}
Let $G$ be a graph such that $\mathsf{d}_{c}(G) = 3$ and let $\{ V_{1} , V_{2},  V_{3}\}$ be any connected domatic partition of $V(G)$. If $G$ is planar, then $\langle V_{1} \rangle$, $\langle V_{2} \rangle$ and $\langle V_{3} \rangle$ are paths.
\end{theorem}

In \cite{7}, Zelinka extended the concept of domatic number to digraphs as follows: for a digraph $D$, an \textit{in-domatic partition of $V(D)$} is a partition of $V(D)$ into in-dominating sets. The maximum number of classes in an in-domatic partition is the \textit{in-domatic number of the digraph $D$}, denoted by $\mathsf{d}^{-}(D)$. Ben\'itez-Bobadilla and Pastrana-Ram\'irez \cite{1} studied this parameter in the Cartesian product of digraphs and some associated digraphs, as the line digraph.

In this paper we extend the concept of connected domatic number to digraphs as follows: for a digraph $D$ a \textit{strong in-domatic partition of $V(D)$} is a partition of $V(D)$ into strong in-dominating sets. The maximum number of classes in a strong in-domatic partition is the \textit{strong in-domatic number of $D$}, denoted by $\mathsf{d}_{s}^{-}(D)$. A strong in-domatic partition of $D$ with $\mathsf{d}_{s}^{-}(D)$ classes is called a \textit{$\mathsf{d}_{s}^{-}$-partition of $V(D)$}. We say that a digraph $D$ is a \textit{strong in-domatic critical digraph} if for every arc $a$ of $D$, $D-a$ is strong and $\mathsf{d}_{s}^{-}(D-a) = \mathsf{d}_{s}^{-}(D) -1$. 

In this paper we show some properties of strong in-domatic partitions in digraphs and some bounds for the strong in-domatic number.

This paper follows the next order: in section 3, some bounds for the strong in-domatic number are given. We prove a characterization of strong in-domatic critical digraphs, and an infinite family of such digraphs will be shown. Also, we  show an extension of Theorem \ref{PlanarT1} and Theorem \ref{PlanarT2} for planar digraphs. In section 4, we study the strong in-domatic number in the Cartesian product and composition of digraphs. As a consequence of the result for composition of digraphs, we will show that, given two natural numbers, say $p$ and $m$, there exists a digraph of order $p$ with strong in-domatic number $m$. Also, we will show that, given a natural number $p \geq 2$ and a divisor of $p$, say $n$, there exists a strong in-domatic critical digraph of order $p$ with strong in-domatic number $n$.  
In section 5, we work these new concepts in certain associated digraphs, as line digraph. We finish the paper with a brief note about the strong out-domatic number.


\section{Terminology and notation} 
 
 For general concepts we refer the rader to  \cite{8} and \cite{4}. Let  $G=(V(G),E(G))$ be a simple undirected graph. An \textit{ isolated vertex} of $G$ is a vertex whose degree is zero. For a nonempty subset  of  $V(G)$, say $S$, the \textit{ subgraph induced by $S$} is denoted by $\langle S \rangle$. If $S$ is such that  $\langle S \rangle$ is complete, then we say that $S$ is a \textit{ clique} of $G$. We say that $S$ is a \textit{ dominating set} if for every $x$ in $V(G) \setminus S$ there exists $z$ in $S$ such that $xz\in E(G)$. We say that a set $S$ of vertices of $G$ is a \textit{dominating clique} if $S$ is a dominating set and $S$ is a clique. The minimum cardinality among all dominating clique of $G$, denoted by $\gamma_{cl}(G)$, is called the \textit{clique domination number of $G$}. For a connected graph $G$, a \textit{vertex-cut} is a set $S$ of vertices of $G$ such that $G-S$ is disconnected. A vertex cut of minimum cardinality is called \textit{minumum vertex-cut} and this cardinality is denoted by $\kappa (G)$. A vertex-cut with $\kappa (G)$ vertices is a \textit{$\kappa $-set of $G$}.

Throughout the paper, $D=(V(D),A(D))$ denotes a loopless digraph with vertex set $V(D)$ and arc set $A(D)$. For an arc ($u$,$v$), $u$ and $v$ are its \textit{end-vertices}; we say that the end-vertices are \textit{adjacent}, we also say that $u$ \textit{out-dominates} $v$ and $v$ \textit{in-dominates}  $u$.
We  say that the arc ($u$,$v$) is \textit{symmetric} if ($v$,$u$) $\in$ $A(D)$.  Let $S$ be a  subset of $V(D)$ and $x$ in $V(D)$, we say that \textit{ $x$ is in-dominated by $S$} (\textit{$x$ is out-dominated by $S$}) if  $z$ in-dominates $x$ for some $z$ in $S$ ($z$ out-dominates $x$ for some $z$ in $S$). We  say that $S$  is  an \textit{in-dominating set} (\textit{out-dominating set}) if every vertex in $V(D) \setminus S$ is in-dominated by $S$ (out-dominated by $S$).  
If $x$ is a vertex of $D$, then the \textit{ ex-neighborhood  of $x$} is the set $\{ z \in V(D) \, : \, (x,z) \in A(D)\}$, denoted by $N^{+}(x)$, the  \textit{ in-neighborhood of $x$} is the set $\{ z \in V(D)  :  (z,x) \in A(D)\}$, denoted by $N^{-}(x)$. The \textit{neighborhood of $x$} is the set $N^{+}(x) \cup N^{-}(x)$ and it is denoted by $N(x)$. The \textit{ out-degree} $\delta_{D}^{+}(x)$ of a vertex $x$ is the number $|N^{+}(x)|$, the \textit{ in-degree} $\delta_{D}^{-}(x)$ of a vertex $x$ is the number $|N^{-}(x)|$. 

A vertex $v$ is called  \textit{isolated} if $\delta_{D}^{+}(v)=0=\delta_{D}^{-}(v)$. For a subset $S$ of $V(D)$,  the \textit{subdigraph of $D$ induced  by  $S$}, denoted by $D\langle S \rangle$,  has $V(D\langle S \rangle)$ = $S$ and $A(D\langle S \rangle)$ = \{($u$,$v$) $\in$ $A(D)$ :  \{$u$, $v$\} $\subseteq$ $S$\}. For a subset $E$ of $A(D)$,  the \textit{ subdigraph of $D$ induced  by the arc set $E$}, denoted by $D[ E ] $,  has $V(D[E]) = \{ x \in V(D) : x\text{ is an end-vertex of some } e \in E \}$ and $A(D[E]) =E$. 

A pair of digraphs $D$ and $H$ are \textit{isomorphic}, denoted by $D \cong H$, if there exists a bijection $f:V(D)\rightarrow V(H)$ such that $(u,v) \in A(D)$ if an only if $(f(u),f(v)) \in A(H)$.  Let  $S_1$ and $S_2$ be subsets of $V(D)$, an arc ($u$,$v$) of $D$ will be called an $S_1S_2$-\textit{ arc} whenever $u$ $\in$ $S_1$ and $v$ $\in$ $S_2$. If $S_1 = \{x\}$ or $S_2 = \{x\}$, then we will write $xS_2$-\textit{ arc} or $S_1x$-\textit{ arc}, respectively.

A \textit{ directed walk} $W$ in $D$ is a sequence of vertices $(x_{0}, x_{1}, \dots , x_{n})$ such that $(x_{i}, x_{i+1}) \in A(D)$ for every $i$ in $\{ 0, 1 \dots , n-1\}$. We  say that $W$ is an $x_0x_n$-\textit{ walk}. 
The \textit{ length} of $W$ is the number $n$. If $x_i$ $\neq$ $x_j$ for all $i$ and  $j$ such that \{$i$, $j$\} $\subseteq$ \{0, $\ldots$ , $n$\} and $i$ $\neq$ $j$,  then $W$ is called  a \textit{ directed  path} ($x_0x_n$-\textit{ path}). 
Let   \{$x_i$, $x_j$\} be a  subset of  $V(W)$, with $i \leq j$, the $x_ix_j$-walk ($x_i$, $x_{i+1}$, $\ldots$ , $x_{j-1}$, $x_j$) contained in $W$ will be denoted by $(x_i,W,x_j)$. If $T=(x_1, \ldots, x_n)$ and $T'=(z_1, \ldots , z_m)$ are directed walks in $D$ and $x_n=z_1$, we donte by $T\cup T'$ the directed walk $(x_1, \ldots , x_n=z_1, z_2, \ldots , z_m)$.
 A  \textit{ directed cycle} is a directed walk ($v_1$, $v_2$,  $\ldots$ , $v_n$, $v_1$) such that $v_i$ $\neq$ $v_j$ for all $i$ and $j$ such that  \{$i$, $j$\} $\subseteq$ \{1,  $\ldots$ , $n$\} and $i$ $\neq$ $j$.
In what follows we  write walk, path and cycle instead of directed walk, directed path and directed cycle, respectively.

 We  say that $D$ is \textit{strong} if, for every pair of vertices $u$ and $v$ in $D$, there exists a $uv$-walk and there exists a $vu$-walk in $D$. If $D$ is a digraph and $S$ is a subset of $V(D)$ we say that $S$ is a \textit{strong in-dominating set} if $S$ is an in-dominating set and $D\langle S \rangle$ is a strong digraph. If $D$ is a strong digraph, a nonempty subset of $A(D)$, say $E$, is a \textit{strong cover of $D$} if $D[E]$ is a spanning strongly connected subdigraph of $D$.  Set $\mathcal{E} = \{ |E| : E \text{ is a partition of } A(D) \text{ into strong covers}\} $, and let  $\Lambda (D)$ be the maximum of $\mathcal{E}$. A partition of $A(D)$ into strong covers with $\Lambda (D)$ elements is called \textit{$\Lambda$-partition of $A(D)$}.

 A digraph is \textit{semicomplete} if for every $u$ and $v$ in $V(D)$ we have that $\{(u,v), (v,u)\} \cap A(D) \neq \emptyset$.  The \textit{ line digraph} of $D$, denoted by $L(D)$, is the digraph such that $V(L(D)) = A(D)$, and (($u$,$v$),($w$,$z$)) $\in$ $A(L(D))$ if and only if $v$ = $w$. 

Let $D$ be a digraph. The \textit{subdivision digraph of $D$}, denoted by $S(D)$, the \textit{root digraph of $D$}, denoted by $R(D)$, the \textit{middle digraph of $D$}, denoted by $Q(D)$, and the \textit{total digraph of $D$}, denoted by $T(D)$, are defined as follows.
$$V(S(D))=V(R(D))=V(Q(D))=V(T(D))=V(D)\cup A(D).$$
And for every vertex $x$ in $V(D)\cup A(D)$.
\begin{equation*}
N_{S(D)}^{+}(x)= \left\lbrace
\begin{aligned}
\{ (u,v)\in A(D) : u=x \} & \text{\, \,  if\, \,  } x \in V(D)\\
\{ v \} & \text{\, \,  if \, \, } x=(u,v) \text{\, \, for some\, \, } (u,v) \in A(D)
\end{aligned}
\right.
\end{equation*}

\begin{equation*}
N_{R(D)}^{+}(x)= \left\lbrace
\begin{aligned}
\{ (u,v)\in A(D) : u=x \} \cup N^{+}_{D}(x) & \text{\, \,  if\, \,  } x \in V(D)\\
\{ v \} & \text{\, \,  if \, \, } x=(u,v) \text{\, \, for some\, \, } (u,v) \in A(D)
\end{aligned}
\right.
\end{equation*}

\begin{equation*}
N_{Q(D)}^{+}(x)= \left\lbrace
\begin{aligned}
\{ (u,v)\in A(D) : u=x \} & \text{\, \,  if\, \,  } x \in V(D)\\
\{ v \} \cup \{ (v, y ) \in A(D) : y \in V(D) \} & \text{\, \,  if \, \, } x=(u,v) \text{\, \, for some\, \, } (u,v) \in A(D)
\end{aligned}
\right.
\end{equation*}
 
\begin{equation*}
N_{T(D)}^{+}(x)= \left\lbrace
\begin{aligned}
\{ (u,v)\in A(D) : u=x \}  \cup N^{+}_{D}(x) & \text{\, \,  if\, \,  } x \in V(D)\\
\{ v \} \cup \{ (v, y ) \in A(D) : y \in V(D) \} & \text{\, \,  if \, \, } x=(u,v) \text{\, \, for some\, \, } (u,v) \in A(D)
\end{aligned}
\right.
\end{equation*}

Notice that $Q(D)$ is a spanning subdigraph of $T(D)$ and $T(D) \langle A(D) \rangle$ is the line digraph of $D$. It is straightforward to see that $S(D)$ and $R(D)$ have no in-dominating vertex if $D$ has at least one arc. 

The \textit{Cartesian product} of two digraphs $D$ and $H$, denoted by $D \square H$, is the digraph whose vertex set  is $V(D) \times V(H)$ and $((x,z),(u,v)) $ is an arc of $D \square H$ if and only if either $x=u$ and $(z,v)\in A(H)$ or $ z=v$ and $(x,u) \in A(D)$. The \textit{ horizontal level} of the vertex $x_{0}$ in $D\square H$ is $H_{x_{0}}=\{ (x_{0}, y) \in V( D \square H ) \, : \, y \in V(H)\}$ and the \textit{ vertical level} of the vertex $y_{0}$ is $D_{y_{0}}=\{ (x, y_{0}) \in V( D \square H ) \, : \, x \in V(D)\}$. Notice that $(D \square H)\langle H_{x_{0}}\rangle \cong H$ and $(D \square H)\langle D_{z_{0}} \rangle \cong D$ for every $x_{0}$ in $V(D)$ and every $z_{0}$ in $V(H)$, respectively.

 Let $D$ a digraph and $\alpha =(D_{v})_{v\in V(D)}$ be a sequence of digraphs which are pairwise vertex disjoint. The \textit{ composition} of $D$ respect to $\alpha$, denoted by $D[\alpha]$, is the digraph obtained from $D$ replacing every vertex $v$ of $D$ by the digraph $D_{v}$ and joining every vertex from $V(D_{v})$ to every vertex in $V(D_{u})$ whenever $(v,u) \in A(D)$.

\begin{rmk}  
\label{remark6}
For every $v$ in  $V(D)$, $D[\alpha] \langle V(D_{v})\rangle=D_{v}$. If $x_{v}$ is an arbitrary vertex in $D_{v}$, then $D[\alpha]\langle \{ x_{v} : v\in V(D) \} \rangle \cong D$.
\end{rmk}   

For a digraph $D$, the \textit{underlying graph of $D$}, denoted by $UG(D)$, is the graph such that $V(UG(D))=V(D)$ and $uv$ is an edge in $UG(D)$ if either $(u,v) \in A(D)$ or $(v,u) \in A(D)$. A digraph is \textit{planar} if $UG(D)$ is a planar graph.

The following results will be useful. 

\begin{lemma}
\label{Pevios Lema3}\cite{75}
A digraph is strong if and only if it has a closed spanning walk.
\end{lemma}

\begin{lemma}\cite{9}
\label{remark3}
Let $D$ be a nontrivial digraph without isolated vertex, with at least one arc, and $L(D)$ its line digraph. $L(D)$ is strong if and only if $D$ is strong.
\end{lemma}

\begin{lemma}\cite{99}
\label{Previos Obs 1}
If $D$ and $H$ are vertex disjoint digraphs, then $D\square H$ is strong if and only if $D$ and $H$ are strong. 
\end{lemma}

It is straightforward to see the following.

\begin{lemma}
\label{Previos Lema 1}
Let $D$ be a strong digraph and $(u,v)$ in $A(D)$. If there exists a $uv$-walk in $D$ which does not contain $(u,v)$, then $D - (u,v)$ is strong.
\end{lemma}

\begin{lemma}
\label{Previos L2}
Let $D$ be a digraph with at least one arc and $L(D)$ its line digraph. If $E \subseteq A(D)$ is nonempty, then $L(D)\langle E \rangle = L ( D [ E ] )$. 
\end{lemma}

\section{First results}

In this section we  prove some properties digraphs with at least one strong in-domatic partition.

\begin{theorem}
\label{First Results Prop 1}
Let $D$ be a digraph. $D$ has a strong in-domatic partition if and only if $D$ is strong.
\end{theorem}
\begin{proof}
Notice that if $D$ is strong, then $\{ V(D) \}$ is a strong in-domatic partition of $V(D)$. 

For the converse, let  $ \mathfrak{ S} = \{ S_{1} ,  \ldots, S_{k} \} $ be a strong in-domatic partition of $D$ and $\{u, v\}$ a subset of $V (D) $. If $\{ u, v \} \subseteq S_{i} $ for some $i$ in $\{1,\ldots , k \}$, then there exists a $uv$-walk contained in $\langle S_{i}\rangle$. If $u \in S_{j}$ and $v \in S_{i} $ with $i \neq j $, since $S_{i}$ is an in-dominating set, there exists $x$ in $S_{i}$ such that $(u, x) \in A(D)$. On the other hand, since $\langle S_{i}\rangle $ is strong, there exists a $x v$-walk contained in  $\langle S_{i}\rangle$, say $C' $, so $C=(u, x) \cup C'$ is a $uv$-walk in $D$, concluding that $D$ is strong.
\end{proof}

It follows from Theorem \ref{First Results Prop 1} that we will consider only strongly connected digraphs. On the other hand, despite determinate whether or not a graph has a domatic partition into $k$ disjoint dominating sets ($k \geq 3$) is a NP-complete problem, the following result shows that, given a strong in-domatic partition of the vertices of a digraph $D$ into $k$ disjoint sets ($1 \leq k \leq \mathsf{d}^{-}_{s}(D)$), it is straightforward to show a strong in-domatic partition into $n$ disjoint sets for every $n \in \{1, \ldots , k \}$.

\begin{proposition}
\label{First Results Lem 2}
If $D$ is a digraph and $\mathfrak{S}= \{ S_{1} , S_{2} , \dots , S_{k} \}$ is a strong in-domatic partition of $D$, then the following holds.

a) $\cup_{i \in I} S_{i}$ is a strong in-dominating set of $D$ for every nonempty subset $I$ of $\{1, \ldots , k\}$.

b)  If $I$ is a nonempty subset of $\{ 1, \dots , k \}$, then $\mathfrak{ S}_{I} = \{ S_{t} \, | \, t \notin I \} \cup \{ \cup_{t \in I} S_{t} \} $ is a strong in-domatic partition of $V(D)$.

c)  For every $n$ in $\{1,\ldots , \mathsf{d}_{s}^{-}(D)\}$  there exists a strong in-domatic partition of $V(D)$, say $\mathfrak{S}$, such that $|\mathfrak{S}|=n$.
\end{proposition}
\begin{proof}
$a)$ Let $I$ be a nonempty subset of $\{1, \ldots , k\}$. Suppose that $I=\{ \alpha_{1} , \ldots , \alpha_{r} \}$ and $S = \cup_{i \in I} S_{i}$. Notice that $\{ S_{\alpha_{1}}, \ldots , S_{\alpha_{r}} \}$ is a in-domatic partition of $V(\langle S \rangle)$ and by Theorem \ref{First Results Prop 1}, $\langle S \rangle$ is strong. On the other hand, since $S_{\alpha_{1}}$ is an in-dominating set of $D$ and $S_{\alpha_{1}} \subseteq S$, then $S$ is also an in-dominating set of $D$, concluding that $S$ is a strong in-dominating set.

$b)$ Since $\mathfrak{S}$ is a partition of $V(D)$, then $\mathfrak{S}_{I}$ is also a partition of $V(D)$. According to $(a)$, every $S \in \mathfrak{S}_{I}$ is a strong in-domatic set of $D$.

$c)$ Consider $\mathfrak{S}= \{ S_{1} , S_{2} , \dots , S_{k} \}$ a $\mathsf{d}_{s}^{-}$-partition of $D$. If  $I = \{ n , n+1 , \dots , k\}$, according to $(b)$, then $\mathfrak{S}_{I}= \{ S_{t} \, | \, t \notin I \} \cup \{ \cup_{t \in I} S_{t} \}$ is a strong in-domatic partition of $D$ such that $|\mathfrak{S} |= n$.
\end{proof}

\subsection{Some bounds}

Zelinka \cite{6} showed that the vertex connectivity number of a graph is an upper bound for the connected domatic number. In the same spirit, we have the following upper bound.

\begin{proposition}
If $D$ is a non-semicomplete strong digraph, then $\mathsf{d}_{s}^{-}(D) \leq \kappa (UG(D))$.
\end{proposition}
\begin{proof}
Let $U$ be a  vertex-cut of $UG(D)$ and $S$ a strong in-dominating set in $D$, we will show that $S\cap U \neq \emptyset$. Suppose that $S\cap U=\emptyset$. Since $D\langle S \rangle$ is strong and $UG(D)-U$ is disconnected, then $S \subseteq V(H)$ for some connected component $H$ of $UG(D)- U$. It follows that no vertex in $V(UG(D)- U)\setminus V(H)$ is in-dominated by $S$ in $D$, which is no possible. Hence $S\cap U \neq \emptyset$.

Let $U$ be a $\kappa $-set in $UG(D)$ and $\mathfrak{S}=\{ S_{1}, \ldots , S_{t}\}$ a $\mathsf{d}_{s}^{-}$-partition of $V(D)$. It follows that $S_{i} \cap U \neq \emptyset$ for every $i$ in $\{ 1, \ldots , t \}$. Hence, $\mathsf{d}_{s}^{-}(D) \leq \kappa (UG(D))$.
\end{proof}

Zelinka \cite{3} proved  that every digraph $D$  holds $\mathsf{d}^{-}(D) \leq \delta^{+}(D)+1$. Corollary \ref{First Results Lem 2} is a consequence of this last result.

\begin{corollary}
\label{First Results Prop 2}
If $D$ is a nontrivial strong digraph, then $\mathsf{d}_{s}^{-}(D) \leq \delta^{+}(D)+1$.
\end{corollary}
\begin{proof}
Let $\mathfrak{S}$ be a $\mathsf{d}_{s}^{-}$-partition of $V(D)$. Since $\mathfrak{S}$ is an in-domatic partition of $V(D)$, then $\mathsf{d}_{s}^{-}(D) \leq \mathsf{d}^{-}(D)$. Hence $\mathsf{d}_{s}^{-}(D) \leq \delta^{+}(D)+1$.
\end{proof}

\begin{rmk}
\label{First results Rmk 1}
For a complete digraph, it is straightforward to see that $\mathsf{d}_{s}^{-}(D)  = \delta^{+}(D)+1 = |V(D)|$.
\end{rmk}

The upper bound showed in Corollary \ref{First Results Prop 2} can be improved in digraphs with no isolated vertices.

\begin{proposition}
\label{First Results Prop 3}
If $D$ is a strong digraph without in-dominating vertex, then $\mathsf{d}_{s}^{-}(D) \leq \delta ^{+} (D)$.
\end{proposition}
\begin{proof}
Let $\mathfrak{S}=\{ S_{1},  \ldots, S_{k} \}$ be a $\mathsf{d}_{s}^{-}$-partition of $V(D)$ and $x$ in $V(D)$ such that $\delta^{+} (x ) = \delta^{+} (D)$. For $i$ in $\{ 1, \ldots , k\}$, if $x \in S_{i}$ then $N^{+}(x)\cap S_{j} \neq \emptyset$ for  every $j$ in $\{ 1, 2, \dots , k \}$ with $j \neq i$. On the other hand, since $x$ is not an in-dominating vertex, then $S_{i}$ is a nontrivial strong set, which implies that $N^{+}(x) \cap S_{i}\neq \emptyset$. Hence, $\mathsf{d}_{s}^{-} (D) = |\mathfrak{S}| \leq \delta ^{+}(x)=\delta ^{+}(D)$.
\end{proof} 

\begin{proposition}
\label{FR.L2}
If $D$ is a strong digraph such that $\mathsf{d}_{s}^{-} (D) = \delta^{+} (D) +1 $, then $D$ has an in-dominating vertex. Moreover, every vertex of minimum out-degree is in-dominating.
\end{proposition}
\begin{proof}
Since $\mathsf{d}_{s}^{-}(D)>\delta^{+}(D)$, it follows from Proposition \ref{First Results Prop 3} that $D$ has an in-dominating vertex. On the other hand, consider a $\mathsf{d}_{s}$-partition of $V(D)$, say $\mathfrak{S}=\{ S_{1}, \ldots , S_{k} \}$, and $x$ in $V(D)$ such that $x$ has minimum out-degree. Let $j$ in $\{ 1, \ldots, k \}$ such that $x \in S_{j}$. Notice that $|S_{j}|=1$, otherwise $S_{j}$ is a nontrivial strong set, so $x$ has at least one out-neighbor in $S_{j}$. Since every $S_{i}$ is an in-dominating set, we conclude that $x$ has an out-neighbor in $S_{i}$ for every $i$ in $\{ 1, \ldots , k \}$ and then $\delta^{+}(x) \geq |\mathfrak{S}|$, but this is no possible, because $\delta^{+}(x)=\delta^{+}(D)$ and $|\mathfrak{S}|=\delta^{+}(D) +1$. Therefore, $S_{j}=\{ x \}$. Hence, $x$ is an in-dominating vertex.
\end{proof}  

\begin{corollary}
\label{FR.P2}
Let $D$ be a strong digraph such that  $\mathsf{d}_{s}^{-}(D)=\delta ^{+}(D) +1$ and $N_{0}$ the set of vertices of minimum out-degree. The following holds:
	\begin{enumerate}[a)]	
	\item $N_{0}$ is an in-dominating set and $D\langle N_{0} \rangle$ is a complete digraph.
	
	\item $\gamma _{cl} (UG(D))\leq |N_{0}|$.
	\end{enumerate}
\end{corollary}
\begin{proof}
It follows from Proposition \ref{FR.L2} that $N_{0}$ is an in-dominating set and $\langle N_{0} \rangle$ is a complete digraph. Hence, $N_{0}$ is a dominating clique in $UG(D)$, which implies that $\gamma _{cl} (UG(D))\leq |N_{0}|$.
\end{proof}

\begin{rmk}
\label{First Results Lema 1}
If $D$ is a digraph and $H$ is a spanning strong subdigraph of $D$, then every strong in-domatic partition of $H$ is also a strong in-domatic partition of $D$. In particular, $\mathsf{d}_{s}^{-} (H) \leq\mathsf{d}_{s}^{-} (D)$.
\end{rmk}

The following proposition shows an upper and a lower bound for a particular case of spanning subdigraphs. It is worth mentioning that Proposition \ref{First Results Prop 6} will be useful in order to define strong in-domatic critical digraphs, which will be characterized in Section 3.2. 

\begin{proposition}
\label{First Results Prop 6}
Let $D$ be a digraph such that $\mathsf{d}_{s}^{-}(D) \geq 2$ and $a$ an arc in $D$. If $D-a$ is strong, then  $\mathsf{d}_{s}^{-}(D)-1 \leq \mathsf{d}_{s}^{-}(D-a) \leq \mathsf{d}_{s}^{-}(D)$.
\end{proposition}
\begin{proof}
 Suppose that $a=(u,v)$. Since $D-a$ is a spanning subdigraph of $D$, by Remark \ref{First Results Lema 1}, $\mathsf{d}_{s}^{-}(D-a) \leq \mathsf{d}_{s}^{-}(D)$. 
On the other hand,  consider $D'=D-a$. Notice that if $\mathsf{d}_{s}^{-}(D)=2$, then $\mathsf{d}_{s}^{-}(D)-1 \leq \mathsf{d}_{s}^{-}(D-a)$ and the first inequality holds. Suppose that $\mathsf{d}_{s}^{-}(D) \geq 3$. Let $\mathfrak{S}=\{S_{1}, S_{2}, \dots, S_{k}\}$ be a $\mathsf{d}_{s}^{-}$-partition of $D$. Consider two cases on $\{u,v\}$.

\underline{Case 1}. $u \in S_{i} $ and $v \in S_{j}$ for some $\{i,j\}$ subset of  $\{ 1, \ldots , k \} $, with $i \neq j$.

Since $r \geq 3$, we can choose an index $r$ in $ \{1,\ldots , k\} \setminus \{i,j\}$ and set  $S_{0}=S_{j} \cup S_{r}$. Consider 
$$\mathfrak{S}' = \{ S_{l} \in \mathfrak{S} \, | \, l \notin \{j,r \}\} \cup \{ S_{0}\}.$$
We claim that $\mathfrak{S}'$ is a strong in-domatic partition of $V(D')$. Let $S$ be an arbitrary element in $\mathfrak{S}'$. If $S \neq S_{0}$, then $S$ is a strong in-dominating set in $D'$. If $S=S_{0}$, since $S_{r}$ is a in-dominating set in $D'$, then $S_{0}$ is also a in-dominating set in $D'$. Moreover, given that $\{S_{j}, S_{r} \}$ is a strong in-domatic partition of $D' \langle S_{0} \rangle$, by Proposition \ref{First Results Prop 1}, $D'\langle S_{0}  \rangle$ is strong, concluding that $\mathfrak{S}'$ is a strong in-domatic partition of $D'$. Therefore $\mathsf{d}_{s}^{-}(D)-1 \leq \mathsf{d}_{s}^{-}(D-a)$.

\underline{Case 2.} $\{ u, v\} \subseteq S_{j} $ for some $j$ in $\{1, \ldots , k \}$.

Let $x_{0}$ be a vertex in $V(D) \setminus S_{j}$, notice that $x_{0} \notin \{u,v\}$ because $k\geq3$. Let $T=(x_{0},  \ldots, x_{n}=v)$ be an $x_{0}v$-walk in $D'$ and consider $\alpha = max \{ i \in \{ 0,\dots , n-1\} \, | \, x_{i} \notin S_{j} \}$. Let $S_{r} \in \mathfrak{S}$ such that $x_{\alpha} \in S_{r}$ and set $S_{0}=S_{j} \cup S_{r}$. Consider $\mathfrak{S}' = \{ S_{l} \in \mathfrak{S}\, | \, l \notin \{ j,r \} \} \cup \{ S_{0} \}.$  
We claim that $\mathfrak{S}'$ is a strong in-domatic partition of $V(D')$.  Let $S$ be an arbitrary element in $\mathfrak{S}'$. If $S \neq S_{0}$, then $S$ is a strong in-dominating set of $D'$. If $S=S_{0}$, since $S_{r}$ is a in-dominating set of $D'$, then $S_{0}$ is also a in-dominating set in $D'$. On the other hand, since $S_{r}$ is a strong in-dominating set in $D'$, there exists $x$ in $S_{r}$ such that $(u,x) \in A(D')$ and there  exists a $xx_{\alpha}$-walk in $D' \langle S_{0} \rangle$, say $T'$, concluding that $(u,x)\cup T'\cup (x_{\alpha}, T, v)$ is a $uv$-walk contained in $D\langle S_{0} \rangle -a$. Therefore, it follows from Lemma \ref{First Results Lem 2} $(a)$ and Lemma \ref{Previos Lema 1} that $D'\langle S_{0} \rangle = D\langle S_{0} \rangle -a$ is strong. Thus, $\mathfrak{S}'$ is a strong in-domatic partition of $V(D')$ and  $\mathsf{d}_{s}^{-}(D)-1 \leq \mathsf{d}_{s}^{-}(D-a)$.
\end{proof}

\subsection{Strong in-domatic critical digraphs}

In \cite{6}, Zelinka defined a conectivelly comatically critical graph as a graph $G$ such that $\mathsf{d}_{c}(G-e) < \mathsf{d}_{c}(G)$ for every edge $e$ of $G$, and the author showed a characterization of such graphs. In the same spirit, we say that a digraph $D$ is a \textit{strong in-domatic critical digraph} if for every arc $a$ of $D$, $D-a$ is strong and $\mathsf{d}_{s}^{-}(D-a) = \mathsf{d}_{s}^{-}(D)-1$. We will show a characterization of such digraphs (Theorem \ref{First Results Prop 7}) and we will give an infinite family of strong in-domatic critical digraphs (Corollary \ref{First Results Prop 8}).

\begin{theorem}
\label{First Results Prop 7}
Let $D$ be a digraph such that $\mathsf{d}_{s}^{-} (D) \geq 2$ and $D-a$ is strong for every $a$ in $A(D)$. The following are equivalent:
\begin{enumerate}[a)]
\item $D$ is a strong in-domatic critical digraph.

\item  If $\mathfrak{S} = \{ S_{1}, \ldots, S_{k} \}$ is a $d_{s}^{-}$-partition of $D$, then $\mathfrak{S}$ holds:
	\begin{enumerate}[b.1)]
	\item $D \langle S_{i} \rangle - a $ is not strong for every $i$ in $\{1,\ldots, k \}$ and  every $a$ in $A ( D\langle S_{i} \rangle )$.

	\item  $|S_{i} \cap N^{+}(u)|=1$ for every $i$ in $\{ 1, \ldots, k \}$ and every $u$ in $V(D) \setminus S_{i}$.
	\end{enumerate}

\end{enumerate}
\end{theorem}
\begin{proof}
Suppose that $D$ is a strong in-domatic critical digraph and let $\mathfrak{S}=\{ S_{1}, \ldots , S_{k} \}$ be a $\mathsf{d}_{s}^{-}$-partition of $V(D)$, we claim that $\mathfrak{S}$ holds (b.1) and (b.2). Proceeding by contradiction, suppose that $\mathfrak{S}$ does not fulfill either (b.1) or (b.2). If $\mathfrak{S}$ does not hold $(b \ldotp 1)$, then there exist $i$ in $\{ 1, \ldots , k \}$ and $a$ in $A( D\langle S_{i} \rangle )$ such that $D \langle S_{i} \rangle -a $ is strong, which implies that $\mathfrak{S}$ is a strong in-domatic partition of $D-a$. Hence, $\mathsf{d}_{s}^{-} (D-a) \geq \mathsf{d}_{s}^{-} (D)$, which is a contradiction since $D$ is a strong in-domatic critical digraph. In the same way, if $\mathfrak{S}$ does not hold $(b\ldotp 2)$, then there exist $i$ in $\{ 1, \ldots , k \}$, a vertex $u$ in $V(D) \setminus S_{i}$ and $\{ x,z \}$ a subset of $S_{i}$ such that $\{ (u,x), (u,z) \} \subseteq A(D)$. In that case, $\mathfrak{S}$ is a strong in-domatic partition of $D-(u,x)$. Hence, $\mathsf{d}_{s}^{-} (D-(u,x)) \geq \mathsf{d}_{s}^{-}(D)$, a contradiction. Therefore, every $d_{s}^{-}$-partition of $D$ holds $(b \ldotp 1)$ and $(b \ldotp 2)$.

Suppose that every $\mathsf{d}_{s}^{-}$-partition of $D$ holds $(b\ldotp1)$ and $(b\ldotp 2)$, but $D$ is not a strong in-domatic critical digraph. Let $(x,z) \in A(D)$ such that $\mathsf{d}_{s}^{-}( D-(x,z) )=\mathsf{d}_{s}^{-} (D)$, and consider a  $\mathsf{d}_{s}^{-}$-partition of $D-(x,z)$, say $\mathfrak{S}= \{ S_{1}, S_{2}, \dots , S_{k} \}$. Since $\mathsf{d}_{s}^{-}(D-(x,z)) = \mathsf{d}_{s}^{-}(D)$, we have that $\mathfrak{S}$ is also a $\mathsf{d}_{s}^{-}$-partition of $D$. If $\{ x,z \}$ is a subset of $S_{i}$ for some $i\in \{ 1, \ldots , k \}$, then $D \langle S_{i} \rangle - (x,z) = (D-(x,z) ) \langle S_{i} \rangle$.  On the other hand, given that $\mathfrak{S}$ is a strong in-domatic partition of $D-(x,z)$, we have that $(D-(x,z) ) \langle S_{i} \rangle$ is strong,  which is no possible by $(b \ldotp 1)$. 
In the case $x \in S_{j}$ and $z \in S_{i}$ for some subset $\{i,j \}$ of $ \{1, \ldots, k\}$ with $i \neq j$, given that $S_{i}$ is a in-dominating set in $D-(x,z)$,  we have that there exists a vertex $w$ in $S_{i} \setminus \{ z \}$ such that $(x,w) \in A(\, D-(x,z)\,)$, a contradiction with $(b\ldotp 2)$. Therefore, $D$ is a strong in-domatic critical digraph.
\end{proof}

The following corollary shows an infinite family of strong in-domatic critical digraphs.

\begin{corollary}
\label{First Results Prop 8}
For $n$ in $\mathbb{N}$, with $n \geq 3$, there exists a strong in domatic critical digraph $D$ with order $2n$ and $\mathsf{d}_{s}^{-}(D)=n$.
\end{corollary}
\begin{proof}
Let $n$ in $\mathbb{N}$ with $n \geq 3$, $U = \{ u_{1} , u_{2} , \dots , u_{n} \}$ and $V = \{ v_{1} , v_{2} , \dots , v_{n} \}$ disjoint sets. Consider the digraph $D$ with vertex set $V \cup U $ and the arc set given by:

$\bullet $ $(u_{i} , u_{j} ) \in A(D)$ if and only if $i < j $.

$\bullet $ $(v_{i} , v_{j} ) \in A(D)$ if and only if $i < j $.

$\bullet $ $(v_{i} , u_{j} ) \in A(D)$ if and only if $i \geq j $.

$\bullet $ $(u_{i} , v_{j} ) \in A(D)$ if and only if $i \geq j $.

We will prove that $D-a$ is strong for every $a \in A(D)$. Consider the following paths:
$$T_{1} = (u_{1}, u_{2}, \dots , u_{n}).$$ 
$$T_{2} = (v_{1}, v_{2}, \dots, v_{n}).$$ 
$$T_{3} = (v_{n}, u_{n}, v_{n-1}, u_{n-1}, \dots, u_{2}, v_{1}, u_{1}).$$
$$T_{4} = (u_{n}, v_{n} , u_{n-1}, v_{n-1}, \dots, v_{2}, u_{1}, v_{1}).$$ 
Given that $C_{1} =T_{1} \cup (u_{n}, v_{1}) \cup T_{2} \cup (v_{n}, u_{1})$ and $C_{2} = (v_{1}, v_{n}) \cup T_{3} \cup (u_{1}, u_{n}) \cup T_{4}$ are spanning closed walks in $D$ that are arc disjoint, it follows that $D-a$ is strong for every $a \in A(D)$.

On the other hand, since the partition $\mathfrak{S}_{0} = \{  \{ u_{i} , v_{i} \} :  i \in \{ 1, \ldots , n \}  \}$ is a strong in-domatic partition of $V(D)$, we have that $n \leq \mathsf{d}_{s}^{-}(D)$. Moreover, since $D$ has no dominating vertex, it follows from Proposition \ref{First Results Prop 3} that $\mathsf{d}_{s}^{-}(D) \leq \delta ^{+} (D)$, concluding that $\mathsf{d}_{s}^{-} (D) = n$ (because $\delta^{+}(u_{1})=n$).
 
Let $\mathfrak{S}= \{S_{1}, \ldots S_{n} \}$ be a $\mathsf{d}_{s}^{-}$-partition of $V(D)$. Given that $D$ has no in-dominating vertex, we conclude that for every $i$ in $\{1, \ldots, n \}$, $|S_{i}|\geq 2$. On the other hand, due to $D$ has order $2n$ and $\mathfrak{S}$ is a partition of $V(D)$ with $n$ elements, we get that for every $i$ in $\{ 1,\ldots , n \}$, $|S_{i}|=2$ . Therefore, for every $i$ in $\{ 1, \ldots , n \}$ we have proved that $S_{i} = \{ u_{j}, v_{j}\}$ for some $j$ in $\{ 1, \ldots, n\}$ and so $\mathfrak{S}= \mathfrak{S}_{0}$.

Since $\mathfrak{S}_{0}$ holds the conditions $(b \ldotp 1)$ and $(b\ldotp 2)$ of Proposition \ref{First Results Prop 8}, we have that $D$ is a strong in-domatic critical digraph.
\end{proof}

\subsection{Strong in-Domatic number in planar digraphs}

In this section, we show the version for digraphs of Theorem \ref{PlanarT1} and Theorem \ref{PlanarT2}, namely Theorem \ref{First Results Prop 9} and Theorem \ref{First Results Prop 10}.  The following result will be useful in order to show an upper bound for the strong in-domatic number in planar digraphs.

\begin{lemma}
\label{First Results Prop 5}
If $D$ is a digraph with at least one strong in-domatic partition, then $UG(D)$ has a connected domatic partition and $\mathsf{d}_{s}^{-} ( D) \leq \mathsf{d}_{c} (UG(D))$.
\end{lemma}
\begin{proof}
 We claim that if $\mathfrak{S}$ is a strong in-domatic partition of $V(D)$, then $\mathfrak{S}$ is a connected domatic partition of $V(UG(D))$. Since $D\langle S  \rangle$ is a strong digraph for every $S$ in $\mathfrak{S}$ then $UG(D)\langle S \rangle$ is a connected graph. On the other hand, since $S$ is an in-dominating set of $D$ for every $S$ in $\mathfrak{S}$, then $S$ is a dominating set in $UG(D)$. Therefore $\mathfrak{S}$ is a connected domatic partition of $V( UG(D) )$. In a particular case, if $\mathfrak{S}$ is a $\mathsf{d}_{s}^{-}$-partition of $v(D)$, then $|\mathfrak{S}| \leq \mathsf{d}_{c}(UG(D)\,)$ and it follows that $\mathsf{d}_{s}^{-}(D) \leq \mathsf{d}_{c}(UG(D)\,)$.
\end{proof}

\begin{theorem}
\label{First Results Prop 9}
If $D$ is a strong planar digraph, then $\mathsf{d}_{s}^{-}(D) \leq 4$. Moreover, $\mathsf{d}_{s}^{-}(D) =4$ if and only if $D$ is a complete digraph with order 4.
\end{theorem}
\begin{proof}
Since $D$ is planar, then $UG(D)$ is a planar graph, that implies that  $\mathsf{d}_{c}( UG(D) ) \leq 4$. Thus, by Proposition \ref{First Results Prop 5}, we get that $\mathsf{d}_{s}^{-}(D) \leq 4$.

On the other hand, suppose that $D$ is a strong planar digraph such that $\mathsf{d}_{s}(D) = 4$.  It follows from Proposition \ref{First Results Prop 5} and Theorem \ref{PlanarT1} that $\mathsf{d}_{c}( UG(D)  )=4$ and so by Theorem \ref{PlanarT1} we get that $UG(D)$ is $K_{4}$, which implies that $D$ is a semicomplete digraph of order 4. Since $\mathsf{d}_{s}^{-}(D) =4$ it follows from Proposition \ref{First Results Prop 2} that $\delta^{+}(D)=3$, concluding that $D$ is a complete digraph.

Suppose that $D$ is a complete  digraph with order 4. It follows from remark \ref{First results Rmk 1}  that $\mathsf{d}_{s}^{-}(D) =4$. 
\end{proof}

\begin{theorem}
\label{First Results Prop 10}
Let $D$ be a planar strong digraph such that $\mathsf{d}_{s}^{-}(D)=3$. If $\mathfrak{S}= \{ S_{1} , S_{2} , S_{3} \}$ is a $\mathsf{d}_{s}^{-}$-partition of $V(D)$, then $\langle S_{i} \rangle$ is a symmetric path in $D$ for every $i$ in $\{ 1, 2, 3\}$.
\end{theorem}
\begin{proof}
 Consider the following cases on the order of $D$.

$\bullet$ \textbf{Case 1.} $D$ has order at least $5$.

In this case, we have from Proposition \ref{PlanarT1} that  $\mathsf{d}_{c}(UG(D))\leq 3$. Hence, by assumption and Proposition \ref{First Results Prop 5}, we conclude that $\mathsf{d}_{c}(UD(D))=3$. 
Is straightforward to see that every $\mathsf{d}_{s}$-partition of $D$, say $\{ V_{1}, V_{2}, V_{3} \}$, is a connected domatic partition in $UG(D)$ and by Theorem \ref{PlanarT2} we conclude that $UG(D) \langle V_{1} \rangle$, $UG(D) \langle V_{2} \rangle$ and $UG(D) \langle V_{3} \rangle$ are paths in $UG(D)$. In that case, since $V_{i}$ is a strong set in $D$ for every $i \in \{1, 2, 3 \}$, then $D\langle V_{1} \rangle$, $D\langle V_{2} \rangle$ and $D\langle V_{3} \rangle$ are symmetric paths in $D$.

$\bullet$ \textbf{Case 2.} $D$ has order at most $4$.

Since $D$ has order at most 4 and $\{ V_{1}, V_{2}, V_{3} \}$ is a $\mathsf{d}_{s}$-partition of $D$, then every set $V_{i}$ has at most two vertices. It follows that  $D\langle V_{1} \rangle$, $D\langle V_{2} \rangle$ and $D\langle V_{3} \rangle$ are symmetric paths in $D$.
\end{proof}

\section{Strong in-domatic number in Cartesian Product and composition of digraphs}

First, we will show  a lower bound of the strong in-domatic number in the Cartesian product. 

\begin{theorem}
\label{NSDIFCOD-PD-Prop2}
If $D$ and $H$ are vertex disjoint strong digraphs, then $$\mathsf{d}_{s}^{-}( D \square H ) \geq \text{\textit{max}} \{ \mathsf{d}_{s}^{-} (D) , \mathsf{d}_{s}^{-} (H) \}.$$
\end{theorem}
\begin{proof}
Suppose without loss of generality that $\mathsf{d}_{s}^{-} ( H ) \leq \mathsf{d}_{s}^{-} (D)$ and  consider  a $\mathsf{d}_{s}^{-}$-partition of $D$, say $\mathfrak{S} = \{ S_{1} ,\ldots , S_{k} \}$. Define $V_{i} = \{ (x,y) \in V( D \square H) : x \in S_{i}\, \}$ and $\mathfrak{V} = \{ V_{1} , V_{2} , \dots , V_{k} \}$. We claim that $\mathfrak{V}$ is a partition of $V(D \square H)$ into strong in-dominating sets.

\begin{enumerate}
\item $\mathfrak{V} $ is a partition of $V (D \square H)$.

 It follows from the fact that  $\mathfrak{S}$ is a partition of $V(D)$.

\item For every $i$ in $\{1, \ldots , k\}$, $V_{i}$ is an in-dominating set.

In order to show that $V_{i}$ is an in-dominating set in $D \square H$, consider $(x,v) \in V( D \square H) \setminus  V_{i} $, and we will show that there exists $(u,z) \in V_{i}$ such that  $( (x,v) ,  (u,z)  ) \in A(D \square H)$. Since $(x,v) \in V( D \square H) \setminus  V_{i} $,  we get that $x \notin S_{i}$. Because of $S_{i}$ is an in-dominating set in $D$, there exists $y \in S_{i}$ such that $(x,y) \in A(D)$. On the other hand, by definition of $V_{i}$ it follows that $(y,v) \in V_{i}$, and by definition of $D \square H$, we have that $( (x,v) ,  (y,v)  ) \in A(D \square H)$. Therefore $V_{i}$ is an in-dominating set for every $i$ in $\{ 1,\dots, k\}$. 

\item For every $i$ in $\{ 1, \ldots , k \}$, $(D\square H) \langle V_{i} \rangle $ is strong . 

Since $D \langle S_{i} \rangle$ and $H$ are strong digraphs, it follows from Lemma \ref{Previos Obs 1} that $D\langle S_{i} \rangle \square H$ is strong. On the other hand, it is straightforward to see that $ D \langle S_{i} \rangle \square H = (D \square H) \langle S_{i} \times V(H) \rangle $,  and $(D \square H) \langle S_{i} \times V(H) \rangle = (D\square H)\langle V_{i} \rangle$, concluding that $D\square H \langle V_{i} \rangle$ is strong.

\end{enumerate}

By the above, $\mathfrak{V}$ is a strong in-domatic partition of $V(D \square H)$. In particular, $|\mathfrak{V}| \leq \mathsf{d}_{s}^{-} (D \square H)$ and by supposition, $\mathsf{d}_{s}^{-}( D \square H ) \geq \text{\textit{max}} \{ \mathsf{d}_{s}^{-} (D) , \mathsf{d}_{s}^{-} (H) \}.$
\end{proof}

The following theorem shows a lower bound for the strong in-domatic number in the composition of digraphs.

\begin{theorem}
\label{NSDIFCOD-CD-Prop2}
Let $D$ be a nontrivial strong digraph and $\alpha$ a sequence of pairwise vertex disjoint digraphs, say $\alpha = ( D_{v} )_{v\in V(D)}$. The composition of $D$ respect to $\alpha$ holds that $\mathsf{d}_{s}^{-}(D[\alpha])\geq   min \{ | V(D_{v} )| : v \in V(D) \}.$
\end{theorem}
\begin{proof} 
Suppose that $V(D)=\{ v_{1}, \ldots , v_{p} \}$, and for every $i \in \{ 1, \ldots , p \}$, let  $\{ x_{1}^{i}, \ldots , x_{n_{i}}^{i}\}$ be the vertex set of $D_{v_{i}}$. Consider $n =  min \{ n_{i} : i \in \{ 1, \ldots , p\} \}$, and for every $k \in \{1, \ldots , n-1\}$, we define $S_{k}=\{ x_{k}^{i} : i\in \{1, \ldots , p\} \}$ and $S_{n}= V(D[\alpha ]) \setminus \cup_{i=1}^{n-1}S_{i}$. 

We denote by $\mathscr{S}$ the set $\{ S_{k} : k \in \{1, \ldots , n\} \}$, and we will show that $\mathscr{S}$ is a strong in-domatic partition of $D[\alpha]$.  Clearly $\mathscr{S}$ is a partition of $V(D[\alpha])$. It  remains to show that for every $i \in \{ 1, \ldots , n \}$, $S_{i}$ is a strong in-dominating set of $D[\alpha]$. 

\textbf{Claim 1.} For every $k$ in $\{ 1, \ldots , n \}$, $S_{k}$ is an in-dominating set in $D[\alpha]$. 

Consider $x_{i}^{j}$ in $V(D[\alpha]) \setminus S_{k}$. Since $v_{j}$ has at least one out-neighbor in $D$, say $v_{t}$ (because $D$ is a nontrivial strong digraph), it follows that for $x_{i}^{j}$ there exists $x_{k}^{t} \in S_{k}$  such that $(x_{i}^{j} , x_{k}^{t}) \in A(D)$, concluding that $S_{k}$ is an in-dominating set in $D[\alpha]$ for every $k$ in $\{ 1, \ldots , n \}$.

\textbf{Claim 2.} For every $k$ in $\{ 1, \ldots , n \}$, $D[\alpha ] \langle S_{k} \rangle$ is strong.

If $k \in \{ 1, \ldots , n-1\}$ it follows from remark \ref{remark6} that $D[\alpha]\langle S_{k} \rangle \cong D$, concluding that $D[\alpha]\langle S_{k} \rangle$ is a strong digraph for every $k \in \{ 1, \ldots, n-1\}$. 

It remains to show that $D[\alpha ] \langle S_{n} \rangle$ is a strong digraph. 
Let  $x_{i}^{j}$ and $x_{s}^{t}$ be two vertices in $S_{n}$. We will denote by  $W$ the set $\{ x_{n}^{i} : i \in \{1, \ldots , p \} \}$ and notice that $W \subseteq S_{n}$. Consider the following cases.

\textbf{Case 1.} $\{ x_{i}^{j} , x_{s}^{t} \} \subseteq W$.

Since $D[\alpha]\langle W\rangle \cong D$ and $D$ is a strong digraph, it follows that there exists an $x_{i}^{j}x_{s}^{t}$-walk in $D[\alpha ]\langle S_{n} \rangle$.

\textbf{Case 2.} $x_{i}^{j} \notin W$ and $x_{s}^{t} \in W$.

Consider an out-neighbor of $v_{j}$ in $D$, say $v_{r}$. By definition of $D[\alpha]$ it follows that $x_{n}^{r}$ is a vertex in $W$ such that $(x_{i}^{j} , x_{n}^{r})$ is an arc in $D[\alpha ]$. By case 1, there exists an $x_{n}^{r}x_{s}^{t}$-walk in $D[\alpha]\langle W \rangle$, say $T$. Hence $(x_{i}^{j}, x_{n}^{r}) \cup T$ is an $x_{i}^{j}x_{s}^{t}$-walk in $D[\alpha]\langle S_{n} \rangle$. 

Now consider an in-neighbor of $v_{j}$ in $D$, say $v_{l}$. In the same way, $x_{n}^{l}$ is a vertex in $W$ such that $(x_{n}^{l}, x_{i}^{j})$ is an arc of $D[\alpha]$. By case 1, there exists an $x_{s}^{t}x_{n}^{l}$-walk in $D[\alpha]\langle W \rangle$, say $T'$. Hence $T' \cup (x_{n}^{l} , x_{i}^{j})$ is an $x_{s}^{t}x_{i}^{j}$-walk in $D[\alpha ]\langle S_{n} \rangle$. 

\textbf{Case 3.} $\{x_{i}^{j}, x_{s}^{t}\} \cap W = \emptyset$. 

Consider an out-neighbor of $v_{j}$ in $D$, say $v_{r}$. It follows from definition of $D[\alpha ]$ that $x_{n}^{r}$ is a vertex in $W$ and $(x_{i}^{j} , x_{n}^{r})$ is an arc of $D[\alpha ]$. By case 2, there exists an $x_{n}^{r}x_{s}^{t}$-walk in $D[\alpha] \langle W \rangle$, say $T$. Hence, $(x_{i}^{j}, x_{n}^{r})\cup T$ is an $x_{i}^{j}x_{s}^{t}$-walk in $D[\alpha ] \langle S_{n} \rangle$. 

It follows from the preceding cases that $D[\alpha ] \langle S_{n} \rangle $ is a strong digraph.

By Claim 1 and Claim 2 we have that $\mathscr{S}$ is a strong in-domatic partition of $V(D[\alpha ])$. Therefore, $\mathsf{d}_{s}^{-}(D[\alpha])\geq   min \{ | V(D_{v} )| : v \in V(D) \}.$
\end{proof}

As a consequence of the previous results, we have the following corollaries.

\begin{corollary}
\label{NSDIFCOD-CD-Cor1}
Let $m$ and $p$ be two natural numbers with $0< m \leq \frac{p}{2}$ and $p\geq3$. Then there exists a digraph $D$ of order $p$ such that $\mathsf{d}_{s}^{-}(D)=m$.
\end{corollary}
\begin{proof}
Let $q$ and $r$ be two natural numbers such that   $m > r \geq 0$ and $p = mq +r $. Notice that $m \leq \frac{p}{2}$ implies that $q\geq 2$. Consider a cycle $H$ of order $q$, say $(v_{1}, \ldots, v_{q}, v_{1})$, and $\alpha$ a sequence of pairwise vertex disjoint digraphs, say $( D_{v_{1}}, \ldots , D_{v_{q}} )$, such that $A(D_{v_{i}})=\emptyset$ for every $i$ in $\{1, \ldots , q \}$, $|V(D_{v_{q}})|= m+r$ and $|V(D_{v_{i}})|=m$ for every $i$ in $\{1, \ldots, q-1\}$.

We claim that the digraph $D$ defined by $H [\alpha ]$ is the desired digraph. It is straightforward to see that $D$ has order $p$. In order to prove that $\mathsf{d}_{s}^{-}(D) \leq m$, we will show that $D$ holds the hypothesis of Proposition \ref{First Results Prop 3}, that is, $D$ has no  in-dominating vertex. Consider the following cases.

 \textbf{Case 1.} $q \geq 3$.
 
 In this case we have that $H$ has at least three vertices, so $D$ has no in-dominating vertex. 
 
 \textbf{Case 2.} $q=2$.
 
 Since $p=mq+r$ and $m \leq \frac{p}{2}$, we have that $r=0$, which implies that $p=mq$. Hence $m \geq 2$, because $p\geq 3$. Therefore $D$ is a bipartite digraph without in-dominating vertex. 
 
By the above, $D$ has no in-dominating vertex and we get from Proposition \ref{First Results Prop 3} that $\mathsf{d}_{s}^{-}(D) \leq \delta^{+}(D)$. 
On the other hand, notice that for every vertex $x$ in $V(D)\setminus V(D_{v_{q-1}})$ we have that $\delta^{+}_{D} (x) = m$, and for every $x$ in $V(D_{v_{q-1}})$ we have that $\delta^{+}_{D} (x) = m +r$, which implies that $\delta ^{+} (D) = m$. Hence, $\mathsf{d}_{s}^{-}(D) \leq m$. 

Finally, by Proposition \ref{NSDIFCOD-CD-Prop2} we have that $m \leq \mathsf{d}_{s}^{-} (D)$, which implies that $\mathsf{d}_{s}^{-}(D)=m$. Therefore, $D$ is the desired digraph.
\end{proof}

\begin{corollary}
\label{NSDIFCOD-CD-Prop3}
Let $p$ and $n$ be two natural number such that $p\geq 2$ and $n\geq 2$. If $n$ divides $p$, then there exists a strong in-domatic critical digraph of order $p$ such that $\mathsf{d}_{s}^{-}(D)=n$.
\end{corollary}
\begin{proof}
Notice that if $p=n$, then the complete digraph of order $p$ is the desired digraph, by Remark \ref{First results Rmk 1}. So we can assume that $p\neq n$. Let $t$ be in $\mathbb{N}$ such that $p=nt$, $H$ a cycle of order $t$, say $(w_{1}, \ldots , w_{t}, w_{1})$, and $\alpha$  a sequence of $t$ pairwise vertex disjoint digraphs, say $\alpha = (D_{w_{1}}, \ldots , D_{w_{t}})$, such that for every $i$ in $\{ 1, \ldots , t\}$ we have that $A(D_{w_{i}})=\emptyset$ and $|V(D_{w_{i}})|=n$. We claim that the digraph $D$ defined by $H[ \alpha]$ is the desired digraph. 

\textbf{Claim 1}. $D$ has order $p$ and $\mathsf{d}_{s}^{-}(D)=n$. 

An analogous proof as in Corollary \ref{NSDIFCOD-CD-Prop3} will show that $D$ has order $p$ and $\mathsf{d}_{s}^{-}(D)=n$. 

On the other hand, in order to show that $D$ is a strong in-domatic critical digraph, we will show that $D$ holds the hypotheses of Theorem \ref{First Results Prop 7}.  

\textbf{Claim 2.} $D-(u,v)$ is strong for every arc $(u,v)$ in $A(D)$. 

Let $(u,v)$ be an arc in $D$. We prove that there exists a $uv$-walk on $D$ which does not contain the arc $(u,v)$ and then, in order to conclude, we use Lemma \ref{Previos Lema 1}. By construction of $D$, there exists $k$ in $\{ 1, \ldots , t\}$ such that $u \in V(D_{w_k})$ and $v \in V(D_{w_{k+1}})$ (indices modulo $t$). On the other hand, for every $i$ in $\{1, \ldots , t \}\setminus \{k, k-1\}$ consider a vertex $x_{i}$ in $D_{w_{i}}$. Since $n\geq 2$, we can choose $x_{k}$ in $V(D_{w_{k}})\setminus \{u\}$ and $x_{k+1}$ in $V(D_{w_{k+1}})\setminus\{v\}$. Notice that $C=(x_{1}, \ldots, x_{k}, x_{k+1}, \ldots, x_{t}, x_{1})$ is a cycle in $D-(u,v)$. So, $(u, x_{k+1})\cup (x_{k+1}, C, x_{1})\cup (x_{1} , C, x_{k}) \cup (x_{k}, v)$ is a $uv$-walk that does not contain the arc $(u,v)$. Therefore $D-(u,v)$ is strong.

\textbf{Claim 3.} If $\mathscr{S}=\{S_{1}, \ldots , S_{n} \}$ is a $\mathsf{d}_{s}^{-}$-partition of $D$ then $|S_{i} \cap V(D_{w_j})|=1$ for every $i$ in $\{1, \ldots, n \}$ and every $j$ in $\{1, \ldots,  t\}$.

Notice that, by definition of $D$, we have that  $t= min \{ l(C): C \text{ is a cylce in } D \}$. From Lemma \ref{Pevios Lema3} we have that $D\langle S_{i} \rangle$ has a spanning closed walk for every $i$ in $\{ 1, \ldots, n \}$ which implies that $D\langle S_{i} \rangle$ has a cycle, concluding that $|V(D\langle S_{i}\rangle)|\geq t$ for every $i$ in $\{ 1, \ldots n \}$. If there exists $k$ in $\{ 1, \ldots , n \}$ such that $|S_{k}|>t$, then 
$$|V(D)|=\sum_{i=1}^{n} |S_{i}|>\sum _{i=1}^{n} t = nt= |V(D)|,$$
which is no possible. So, every element of $\mathscr{S}$ has $t$ vertices of $D$. 

Since $D\langle S_{i} \rangle$ has a spanning closed walk (by Lemma \ref{Pevios Lema3}), then $D\langle S_{i} \rangle$ contains a cycle $C$ which has length at least $t$. Because of $|S_{i}|=t$ we conclude that $l(C) =t$; that is $D\langle S_{i} \rangle =C$. 

Hence $|S_{i} \cap V(D_{w_j})|=1$ for every $i$ in $\{1, \ldots, n \}$ and every $j$ in $\{1, \ldots,  t\}$.

\textbf{Claim 4.} $|S_{i} \cap N^{+}(x)|=1$ for every $i$ in $\{ 1, \ldots, n \}$ and every $x$ in $V(D) \setminus S_{i}$.

Let $x$ be in $V(D) \setminus S_{i}$ and suppose that $x \in V(D_{w_{k-1}})$ for some $k$ in $\{2, \ldots ,t+1\}$. Since $S_{i}$ is an in-dominating set in $D$, then there exists a  vertex $z$ in $S_{i}$ such that $(x,z) \in A(D)$. Notice that $z \in V(D_{w_{k}})$. Since $|S_{i} \cap V(D_{w_k})|=1$ and $N^{+}(x) \subseteq V(D_{w_k}),$ then $S_{i} \cap N^{+}(x)=\{ z\}$, that is $|S_{i} \cap N^{+}(x)|=1$.

\textbf{Claim 5.}  $D \langle S_{i} \rangle - a $ is not strong for every $i$ in $\{1,\ldots, k \}$ and  every $a$ in $A ( D\langle S_{i} \rangle )$.

Since $D\langle S_{i} \rangle$ is a cycle, we have that $D\langle S_{i} \rangle-a$ is not strong for every arc $a$ in $D\langle S_{i} \rangle$ and every $i \in \{1, \ldots , n \}$.

Therefore, it follows from Claims 1,2,4 and 5, and by Theorem \ref{First Results Prop 7}, that $D$ is a strong in-domatic critical digraph. 
\end{proof}

\section{Strong in-domatic number in line digraph and other associated digraphs}

\begin{proposition}
\label{NSDIFCDA-DL-Prop3}
Let $D$ be a nontrivial strong digraph, $L(D)$ its line digraph and $E$ a nonempty subset of $A(D)$. $E$ is a strong cover of $D$ if and only if $E$ is a strong in-dominating set of $L(D)$. 
\end{proposition}
\begin{proof}
For the sufficiency, consider a strong cover of $D$, say $E$. Since $D [ E ]$ is strong, it follows that $L(D [ E ])$ is strong (by Lemma \ref{remark3}), which implies that $L ( D) \langle E \rangle$ is strong (by Lemma \ref{Previos L2}).   

In order to prove that $E$ is an in-dominating set in $L(D)$, consider a vertex in $V(L(D)) \setminus E$, say $(x,z)$.  Since $E$ is a strong cover, we get that $D [ E ]$ is a spanning subdigraph of $D$, which implies that $z\in V(D [ E ] )$. Because of $D [ E ]$ is a non trivial strong subdigraph, there exists $w$ in $N^{+}_{D[E]}(z)$. Hence $(x,z)$ is in-dominated by $(z,w)$ in $L(D)$. So, $E$ is a strong in-dominating set in $L(D)$.

For the necessary condition of Proposition \ref{NSDIFCDA-DL-Prop3}, suppose that $E$ is a strong in-dominating set in $L(D)$. In order to prove that $D[E]$ is spanning subdigraph of $D$, consider a vertex $x$ in $V(D)$. Since $D$ is a nontrivial strong digraph, then there exists $u$ in $N^{-}_{D}(x)$. If $(u,x) \in E$, then we have that $x \in V( D [ E ] )$. If $(u,x) \in V(L(D))\setminus E$,  since $E$ is an in-dominating set in $L(D)$, then there exists $b$ in $E$ such that that $((u,x),b) \in A(L(D))$. Hence, it follows from definition of $L(D)$ that $b=(x,z)$ for some $z$ in $V(D)$, concluding that $x \in V(D [ E ] )$. Therefore, $D [ E ]$ is a spanning subdigraph of $D$.

We will prove that $D[ E ]$ is strong. Since $L(D) \langle E \rangle$ is strong, we have that  $L(D [ E ])$ is strong (by Lemma \ref{Previos L2}). So, $D [ E ]$ is strong (by Lemma \ref{remark3}). 

Therefore, $E$ is a strong cover of $D$.
\end{proof}

\begin{theorem}
\label{NSDIFCDA-DL-Teo1}
If $D$ is a strong digraph of order at least three, then  $\mathsf{d}_{s}^{-}( L(D) )= \Lambda (D).$
\end{theorem}
\begin{proof}
Consider a $\Lambda$-partition of $A(D)$, say  $\mathscr{F}$. According to Proposition \ref{NSDIFCDA-DL-Prop3} we have that $F$ is an in-dominating strong set in $L(D)$ for every $F$ in $\mathscr{F}$, which implies that $\mathscr{F}$ is a strong in-domatic partition of $V(L(D))$, and so $\Lambda (D) \leq \mathsf{d}_{s}^{-}( L(D) )$. In the same way, if $\mathscr{S}$ is a $\mathsf{d}_{s}^{-}$-partition of $V( L(D)  )$, it follows from Proposition \ref{NSDIFCDA-DL-Prop3} that $S$ is a strong cover of $D$ for every $S$ in $\mathscr{S}$, concluding that $\mathscr{S}$ is a partition of $A(D)$ into strong covers, so $\mathsf{d}_{s}^{-}( L(D)  ) \leq \Lambda (D)$. Therefore, $\mathsf{d}_{s}^{-}(\, L(D) \, )= \Lambda (D)$.
\end{proof}

\begin{lemma}
\label{LemSR}
If $D$ is a strong digraph, then $\mathsf{d}_{s}^{-}(S(D))=\mathsf{d}_{s}^{-}(R(D))=1.$
\end{lemma}
\begin{proof}
If $D$ has order 1, Lemma \ref{LemSR} holds. If $|V(D)| \geq 2$, then for every arc $a$ in $A(D)$ we have that $|N^{+}_{S(D)}(a)|=1$. It follows from definition of $S(D)$ that it has no in-dominating vertex, hence $\mathsf{d}_{s}^{-}(S(D))=1$ (by Proposition \ref{First Results Prop 3}). A similar proof shows that $\mathsf{d}_{s}^{-}(R(D))=1$.
\end{proof}

\begin{proposition}
\label{NSDIFCDA-DA-Prop2}

If $D$ is a nontrivial strong digraph, then $\mathsf{d}_{s}^{-}( L(D) ) \leq \mathsf{d}_{s}^{-}(  Q(D) ).$
\end{proposition}
\begin{proof}
Let $\mathscr{U}$ be a $\mathsf{d}_{s}^{-}$-partition of $V(L(D))$, where $\mathscr{U}= \{ S_{1},  \ldots, S_{k} \}$. Consider the set $\mathscr{S}= \{ S_{1}\cup V(D), S_{2}, \ldots   , S_{k} \}$. Since $V(Q(D))=V(D) \cup A(D)$ we have that $\mathscr{S}$ is a partition of $V(Q(D))$. We will show that every set in $\mathscr{S}$ is a strong in-dominating set in $Q(D)$.

\textbf{Claim 1.} Every set in $\mathscr{S}$ is an in-dominating set in $Q(D)$.

First, we will show that Claim 1 holds for the set $S_{1} \cup V(D)$. Let $x$ be a vertex in $V(Q(D))\setminus ( S_{1}\cup V(D))$. Since $x \in A(D)$ and $S_{1}$ is an in-dominating set in $L(D)$, then there exists $b$ in $S_{1}$ such that $(x,b) \in A(L(D))$; since $A(L(D))) \subseteq A(Q(D))$, then $(x,b) \in A(Q(D))$. Therefore $S_{1} \cup V(D)$ is an in-dominating set in $Q(D)$. 

Now we will show that Claim 1 holds for $S_{i}$, where $i \in \{ 2, \ldots , k\}$. Let $x$ be a vertex in $V(Q(D))\setminus S_{i}$.  If $x \in A(D)$, since $S_{1}$ is an in-dominating set, then there exists $b$ in $S_{1}$ such that $(x,b) \in A(L(D))$; because of $A(L(D)) \subseteq A(Q(D))$, we get that $(x,b) \in A(Q(D))$. 

Suppose that $x \in V(D)$. Since $S_{i}$ is a strong in-dominating set in $L(D)$ it follows from Proposition \ref{NSDIFCDA-DL-Prop3} that $S_{i}$ is a strong cover of $D$. So, there exists an arc $a$ in $S_{i}$ such that $a=(x,z)$ for some $z$ in $V(D)$, which implies that $(x,a) \in A(Q(D))$. 

Therefore, every set in $\mathscr{S}$ is an in-dominating set in $Q(D)$.

\textbf{Claim 2.} For every set $W$ in $\mathscr{S}$, $Q(D) \langle W \rangle$ is strong.

Since $L(D)$ is an induced subdigraph of $Q(D)$ and $L(D) \langle S_{i} \rangle$ is a strong digraph for every $i$ in $\{ 1, \ldots k \}$, then we have that   $Q(D) \langle S_{i} \rangle$ is strong for every $i$ in $\{1, \ldots , k\}$. It remains to prove that $Q(D) \langle S_{1}\cup V(D) \rangle$ is a strong.

Let $x$ and $z$ be two vertices in $S_{1} \cup V(D)$. Consider the following three cases.
\begin{itemize}
\item \textbf{Case 1.} $\{ x,z\}\subseteq S_{1}$. 

Since $L(D) \langle S_{1} \rangle$ is strong, then there exists an $xz$-walk contained in $L(D) \langle S_{1} \rangle$. Since $L(D)$ is a subdigraph of $Q(D)$, we have that there exists an $xz$-walk contained in $Q(D) \langle S_{1} \rangle$.

\item \textbf{Case 2.} $x \in V(D)$ and $z\in S_{1}$.

Since $S_{1}$ is a strong cover of $D$ (by Proposition \ref{NSDIFCDA-DL-Prop3}), we get that there exist arcs $a$ and $b$ in $S_{1}$ such that $a=(v,x)$ and $b=(x,u)$ for some $v$ and $u$ in $V(D)$. It follows from Case 1 that there exists an $za$-walk contained in $Q(D) \langle S_{1} \rangle$, say $C_{1}$. Therefore, $ C_{1} \cup (a,x)$ is a $zx$-walk in $Q(D) \langle S_{1}\cup V(D) \rangle $. In the same way we can prove that there exists a $bz$-walk in $Q(D)\langle S_{1}\rangle$, say $C_{2}$, which implies that $(x,b)\cup C_{2}$ is an $xz$-walk in $Q(D)\langle S_{1} \cup V(D) \rangle$.

\item \textbf{Case 3.} $\{x, z\} \subseteq V(D)$.

Consider $a$ in $S_{1}$. By Case 2 we have that there exists an $xa$-walk contained in $Q(D) \langle S_{1} \cup V(D) \rangle $, say $C_{1}$, and there exists an $az$-walk contained in $Q(D) \langle S_{1}\cup V(D) \rangle$, say $C_{2}$. It follows that $C_{1} \cup C_{2}$ is an $xz$-walk contained in $Q(D) \langle S_{1} \cup V(D) \rangle$.
\end{itemize}
Therefore, $\mathscr{S}$ is a strong in-domatic partition of $V(Q(D))$. In particular, $|\mathscr{S}|\leq \mathsf{d}^{-}_{s}(Q(D))$. Thus, $\mathsf{d}_{s}^{-}(L(D))\leq \mathsf{d}_{s}^{-}(Q(D))$.
\end{proof}

\begin{proposition}
\label{NSDIFCDA-DA-Prop3}
If $D$ is a strong digraph of order at least three, then $\mathsf{d}_{s}^{-}( L(D)  ) +1 \leq \mathsf{d}_{s}^{-}(  T(D)  ).$
\end{proposition}
\begin{proof}
Let $\mathscr{S}'$ be a $\mathsf{d}_{s}^{-}$-partition of $V(L(D))$, say $\mathscr{S}'= \{ S_{1}, \ldots , S_{k} \}$, and consider $\mathscr{S}= \{ V(D) , S_{1}, \ldots , S_{k} \}$. We claim that $\mathscr{S}$ is a strong in-domatic partition of $V(T(D))$. It follows from definition of $T(D)$ that  $\mathscr{S}$ is a partition of $V(T(D))$. It only remains to show that every element in $\mathscr{S}$ is a strong in-dominating set in $T(D)$.

\textbf{Claim 1.} For every set $W$ in $\mathscr{S}$, $T(D)\langle W \rangle$ is strong.

Since $L(D) \langle S_{i}\rangle$ is strong for every $i$ in $\{1, \ldots , k \}$, and $L(D) \langle S_{i} \rangle$ is a spanning subdigraph of $T(D)\langle S_{i} \rangle$, then $T(D) \langle S_{i} \rangle$ is a strong digraph. On the other hand, since $D$ is strong, then $ T(D) \langle V(D) \rangle$ is strong. Therefore, for every set $W$ in $\mathscr{S}$, $T(D)\langle W \rangle$ is strong.

\textbf{Claim 2.} Every element in $\mathscr{S}$ is an in-dominating set in $T(D)$.

We will prove that $V(D)$ is an in-dominating set in $T(D)$. Consider a vertex $x$ in $V(T(D))\setminus V(D)$, it follows that $x=(u,v)$ for some $u$ and $v$ in $V(D)$ and, by definition of $T(D)$, we conclude that $(x, v) \in A(T(D))$. Therefore, $V(D)$ is an in-dominating set in $T(D)$. 

On the other hand, we will prove that $S_{i}$ is an in-dominating set for every $i$ in $\{1, \ldots , k\}$. Let $S_{i}$ in $\mathscr{S}$ and $x$ in $V(T(D)) \setminus S_{i}$ for some $i$ in $\{1 ,  \ldots , k\}$. If $x\notin V(D)$, since $S_{i}$ is an in-dominting set in $L(D)$, then $x$ is in-dominated by $S_{i}$. If $x\in V(D)$, since $S_{i}$ is a strong cover of $D$ (by Proposition \ref{NSDIFCDA-DL-Prop3}), it follows that there exists an arc $a$ in $S_{i}$ such that $a=(x,v)$ for some $v$ in $V(D)$, which implies that $x$ is in-dominated by $S_{i}$. Therefore, every element in $\mathscr{S}$ is an in-dominating set in $D$.

Since $\mathscr{S}$ is a strong in-dominating partition of $V(T(D))$ we have that $|\mathscr{S}| \leq \mathsf{d}_{s}^{-}(  T(D)  )$, which implies that $\mathsf{d}_{s}^{-}( L(D)  ) +1 \leq \mathsf{d}_{s}^{-}(  T(D)  ).$
\end{proof}

\subsection{A note on strong out-domatic number}

Let $D$ be a digraph, the \textbf{converse of $\boldsymbol{D}$}, denoted by $\overset{\leftarrow}{D}$, is the digraph such that $V(\overset{\leftarrow}{D})=V(D)$ and $(u,v) \in A(\overset{\leftarrow}{D})$ if and only if $(v,u) \in A(D)$. Notice that if $S$ is an in-dominating set in $V(D)$, then for every vertex $x$ in $V(\overset{\leftarrow}{D}) \setminus S$ there exists $w$ in $S$ such that $(w,x) \in A(\overset{\leftarrow}{D})$.  
Therefore, we can consider the following definition; an \textit{out-domatic partition of $V(D)$} is a partition of $V(D)$, say $\mathfrak{S}$, such that for every $S$ in $\mathfrak{S}$, $D\langle S \rangle$ is a strong digraph and every vertex not in $S$ has at least one in-neighbor in $S$. Notice that $\mathfrak{S}$ is a strong in-domatic partition of $V(D)$ if and only if $\mathfrak{S}$ is a strong out-domatic partition of $V(\overset{\leftarrow}{D})$.  The maximum number of elements in an out-domatic partition is called the \textit{strong out-domatic number of $D$} and it is denoted by $\mathsf{d}_{s}^{+}(D)$. It is straightforward to see that $\mathsf{d}_{s}^{-}(D)=\mathsf{d}_{s}^{+}(\overset{\leftarrow}{D})$.

\end{document}